\documentclass[12pt]{article}
      \usepackage{latexsym}
         \usepackage[reqno, namelimits, sumlimits]{amsmath} 
          \usepackage{amssymb, amsthm, amsfonts} 
\usepackage{graphicx}

\newtheorem{Theorem}{Theorem}[section]

\newtheorem{theorem}[Theorem]{Theorem}
\newtheorem{lemma}[Theorem]{Lemma}

\DeclareMathOperator{\diag}{diag}

\newcommand{\R}{{\mathbb R}}
\newcommand{\N}{{\mathbb N}}
\newcommand{\Z}{{\mathbb Z}}

\newcommand{\al}{\alpha}
\newcommand{\be}{\beta}
\newcommand{\ga}{\gamma}

\newcommand{\Ga}{\Gamma}
\newcommand{\de}{\delta}
\newcommand{\ep}{\epsilon}

\newcommand{\De}{\Delta}

\newcommand{\Om}{\Omega}

\newcommand{\la}{\lambda}

\newcommand{\rest}{\big\arrowvert}
\newcommand{\upto}{\nearrow}

\numberwithin{equation}{section}

\title{Nonnegative solutions with a nontrivial
  nodal set for elliptic equations on smooth symmetric domains}
\author{P. Pol\'a\v cik\footnote{Supported in part by NSF grant DMS-0900947}
\\
\small School of Mathematics, University of Minnesota\\
\small Minneapolis, MN 55455, USA
\\~\\
Susanna Terracini\footnote{Supported in part by PRIN2009 grant ``Critical Point Theory and Perturbative
Methods for Nonlinear Differential Equations''}\\
\small  Dipartimento di Matematica e Applicazioni,\\  \small Universit\`a
 di Milano-Bicocca\\
\small  Piazza Ateneo Nuovo 1,  20126  Milano, Italy
}

\date{} 

\begin{document}
 
\maketitle

\begin{quote}{\small
    {\bfseries Abstract.} 
We consider a semilinear elliptic equation on a 
smooth bounded domain $\Om$ in $\R^2$, assuming that
both the domain and the equation are invariant under reflections about
one of the  coordinate axes, say the $y$-axis. It
is known that nonnegative solutions of the Dirichlet problem for such
equations are symmetric about the  axis, and, if  strictly
positive, they are also decreasing in $x$ for $x>0$. Our goal is to
exhibit examples of equations which admit 
nonnegative, nonzero solutions for which the second property
fails; necessarily, such solutions have a nontrivial nodal set in
$\Om$. Previously, such examples were known for nonsmooth domains only.
} 
 \end{quote}



 \bigskip

\section{Introduction and  the main result}\label{intro}

Consider the elliptic  problem
   \begin{alignat}{2}
    \De u+f(x',u)&=0,&\quad &x\in\Om,\label{eq}\\
    u&=0,&\quad &x\in\partial\Om,\label{bc}
\end{alignat}
where $\Om$ is a  bounded domain in $\R^N$, $x=(x_1,x')\in
\R\times\R^{N-1}$, 
and  $f:\R^{N-1} \times \R\to\R$  is a continuous function
which is locally Lipschitz  in the last variable. 
We assume that $\Om$ is convex in $x_1$ and  
reflectionally symmetric about the 
hyperplane  
$$H_0=\{(x_1,x')\in \R\times\R^{N-1}:x_1=0\}.$$ 
By a well-known theorem of  Gidas, Ni, and Nirenberg
\cite{Gidas-N-N:bd} and its more general versions for nonsmooth
domains, as
given by Berestycki and
Nirenberg \cite{Berestycki-N} and Dancer \cite{Dancer:symm},
each positive solution $u$ of
\eqref{eq}, \eqref{bc}  is
even in $x_1$:
\begin{equation}
  \label{symm}
  u(-x_1,x')=u(x_1,x')\quad((x_1,x')\in\Om),
\end{equation}
and, moreover, $u(x_1,x')$ decreases with increasing $|x_1|$:
\begin{equation}
  \label{mon}
  u_{x_1}(x_1,x')<0\quad((x_1,x')\in\Om, x_1>0).
\end{equation}
It is also well-known  that this result is not valid in general 
for nonnegative solutions; consider, for example, the solution
$u(x)=1+\cos x$ of  the equation $u''+u-1=0$ on the interval $\Om=(-3\pi,3\pi)$.
However, as recently proved in \cite{P:symm-ell},
nonnegative solutions still enjoy the symmetry property
\eqref{symm}. Of course, by the Dirichlet boundary condition,
\eqref{mon}  necessarily fails unless the
solution is strictly positive in $\Om$. A further investigation in
\cite{P:symm-ell} revealed that the nodal set of each nonnegative
solution $u$ has interesting symmetry properties itself. In particular,
each nodal domain of $u$ is convex in $x_1$ and symmetric about a
hyperplane parallel to $H_0$ (a nodal domain refers to a  
connected component of $\{x\in \Om: u(x)\ne 0\}$).
These results, like those in \cite{Berestycki-N}, are valid for
fully nonlinear elliptic equations
 \begin{equation}
    F(x,u,D u,D^2u)=0,\quad x\in\Om,\label{eqfully}
\end{equation}
under suitable symmetry assumptions,
 and their proofs employ the
method of moving hyperplanes \cite{Alexandrov, Serrin:symm}
as the basic geometric technique. Related results can be found in 
the surveys  \cite{Berestycki:survey, Kawohl:survey, Ni:survey,
  P:symm-survey}, monographs \cite{Du:bk, Fraenkel:bk, Pucci-S:bk}, or
more recent papers \cite{Brock:symm, Dolbeault-F:symm, Da-L-S,
  Foldes:bd}, among others.

In this work, we are concerned with  nonnegative
solutions which do have a nontrivial nodal set in $\Om$.  
In \cite{P:symm-ell}, several examples of problems \eqref{eq},
\eqref{bc} admitting such solutions were given, including explicit 
examples with solutions whose nodal set consists of  line segments,
 as well as a more involved construction with non-flat nodal curves. 
In all these examples,  the domain $\Om$ has corners and it is not even
of class $C^1$. On the other hand, it was
 also proved in \cite{P:symm-ell} that for some $C^1$ domains $\Om$
satisfying additional geometric conditions, no solutions with
 nontrivial nodal sets can exist, no
 matter how the nonlinearity $f=f(x',u)$ is chosen. For example,
this is the case if $\Om$ is a $C^1$  planar domain such that  the
``cups'' $\{(x_1,x_2)\in \Om:x_1>\la\}$, $\la\in \R$, are all connected and
$\partial \Om$  contains a line segment parallel to the $x_2$-axis. 
 These observations  raise the following natural question:

 \medskip

 \noindent
 {\bf Question.}  \emph{
Does  smoothness 
of $\Om$ alone preclude the existence of nonnegative solutions with  nontrivial
nodal sets for problem
\eqref{eq}, \eqref{bc}?} 

\medskip

The aim of  this paper is to  show that  the answer is negative even 
when analyticity of  both the domain and the nonlinearity is assumed.
We stress that the dependence of the nonlinearity on $x'$  is essential here. 
Indeed, a theorem from \cite{P:symm-serrin} states that when the class 
of equations is restricted to  the homogeneous ones, $\De u+f(u)=0$, and $\Om$ is smooth, then each nonnegative solution is either identical to 0 (and hence $f(0)=0$)
or strictly positive.

Our construction is in two dimensions, hence we  
use the coordinates $(x,y)$ instead of $(x_1,x')$.  
We consider affine equations of the form 
\begin{alignat}{2}
    \De u+4 u+h(y) &=0,&\quad &(x,y)\in\Om,\label{eeq}\\
    u&=0,&\quad &(x,y)\in\partial\Om.\label{ebc}
\end{alignat}
Here is our main result:
\begin{theorem}\label{thm}
  There exist a continuous function
  $h:\R\to\R$ and a bounded analytic domain $\Om\subset \R^2$, 
which is convex in $x$ and symmetric about  the vertical axis 
$H_0=\{(0,y): y\in \R\}$, such that problem
\eqref{eeq}, \eqref{ebc} has a nonnegative solution $u$ whose nodal
set in $\Om$ consists of two analytic curves.
\end{theorem}
By an analytic domain we mean a Lipschitz domain
whose boundary is an analytic submanifold of $\R^2$. 
A curve refers to a one-dimensional submanifold of $\R^2$, possibly
with boundary.

 The nodal set of the solution
$u$, including the boundary of $\Om$, is plotted in Figure~\ref{fig1}. In
accordance with \cite[Theorem 2.2]{P:symm-ell}, 
each   nodal domain of $u$ is symmetric about a line parallel to
the $y$-axis, as indicated by dashed lines in Figure~\ref{fig1},  and the
solution $u$ is  symmetric about that line within the nodal domain.

         \begin{figure}[h]\label{fig1}
         \addtolength{\belowcaptionskip}{10pt}
         \begin{center}
         \includegraphics[scale=.2]{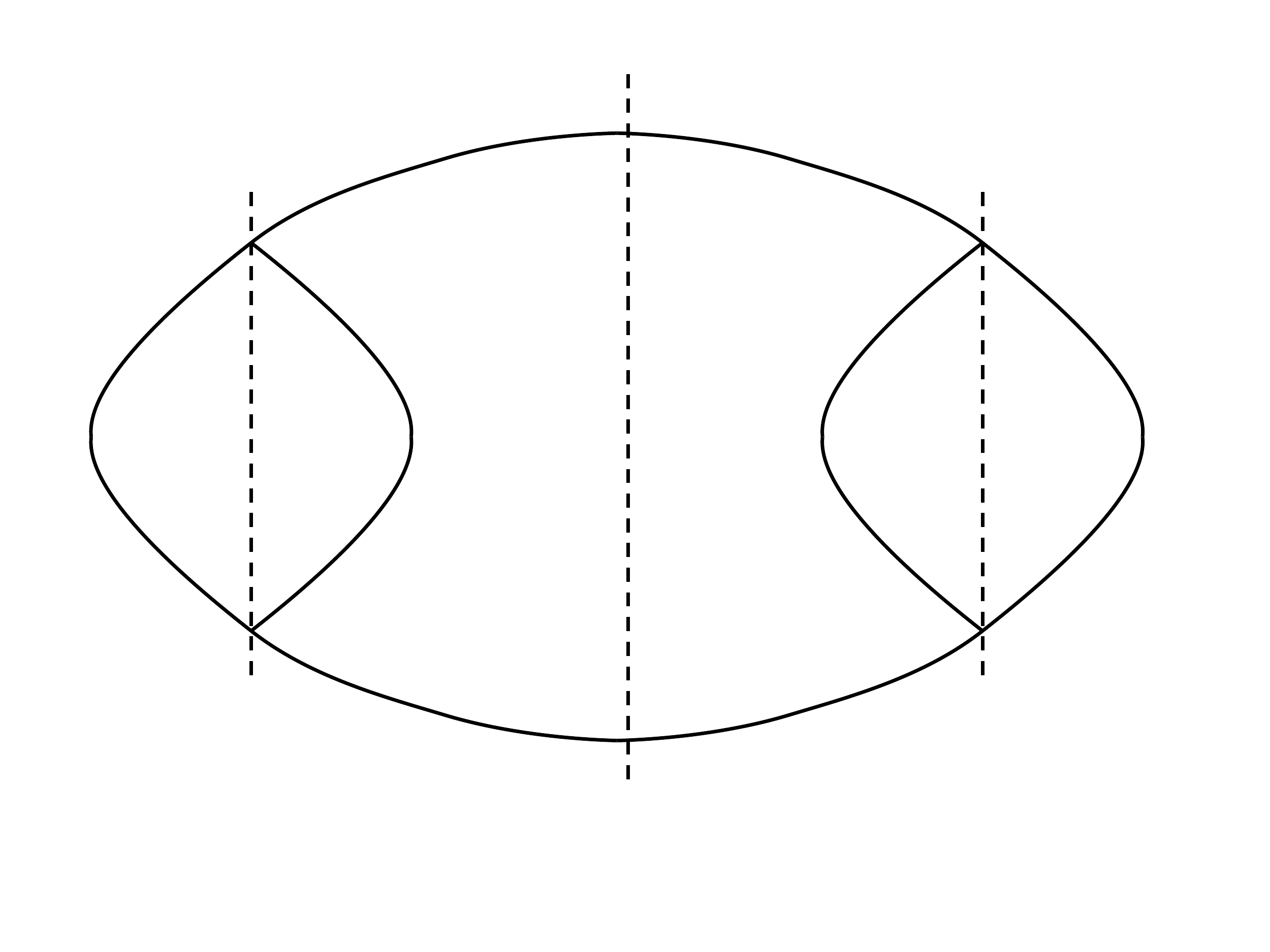}
         \caption[The nodal set of $u$]{\small The nodal set (the solid curves) of $u$.\label{fig1}}   
         \end{center}
         
         \end{figure}

\begin{figure}[h]
         \addtolength{\belowcaptionskip}{10pt}
        %
       \begin{center} \includegraphics[scale=.2]{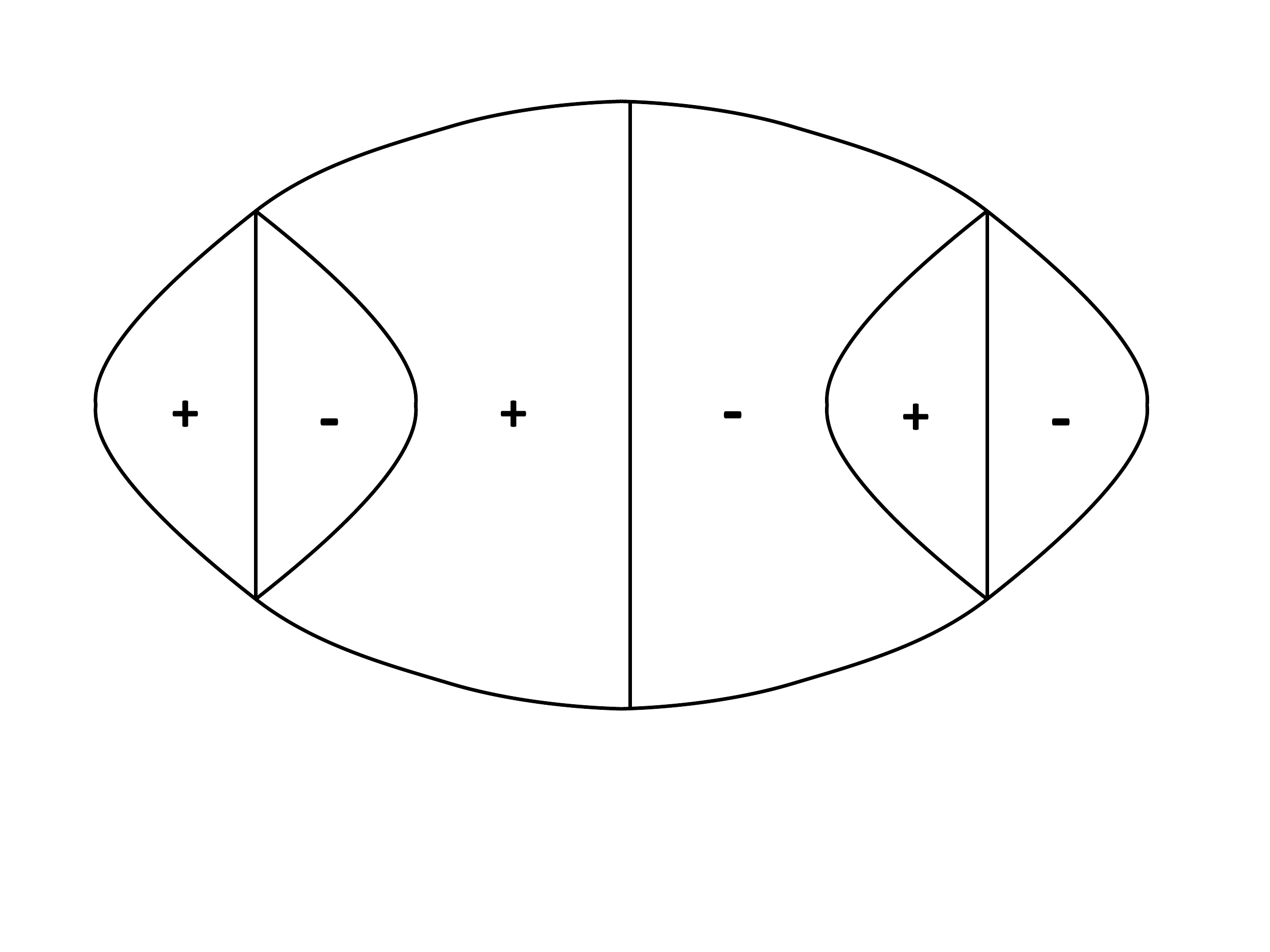}
         \caption[The nodal set $v$]{\small The nodal set and the
           signs  of $v$.\label{fig2}}  
          \end{center}
         \end{figure}

 Similarly to \cite{P:symm-ell}, our construction links the solutions
 of \eqref{eeq} to some eigenfunction of the Laplacian. Specifically,
 if   
 $u$ is a  solution of \eqref{eeq}, then $v=u_x$ satisfies 
 $\De v+4 v=0$ in $\Om$. Moreover, if $u\ge 0$ in $\Om$, then $v=0$ on all 
 nodal curves of $u$ in $\Om$. Also, one has $v=0$ on $H_0$ and all the
 other  symmetry lines of $u$ parallel to $H_0$. Thus, 
 from Figure~\ref{fig1} we  infer that the nodal set of  $v$ should
 look like as indicated  in Figure~\ref{fig2}.

 Thus a key prerequisite for our construction is 
 a solution of $\De v+4 v=0$ with the nodal structure 
 as in Figure~\ref{fig2}. Once such a solution $v$ is found,
 we complete the construction by exhibiting an antiderivative of $v$ with 
respect to $x$ which satisfies \eqref{eeq}, \eqref{ebc} for some function
$h$, is nonnegative,  and has the nodal set as in Figure~\ref{fig1}.

To indicate how we find  a solution $v$  of $\De v+4 v=0$ with the
desired nodal
structure, let us first consider  a  solution of the same equation
given explicitly by  
$$w(x,y):=(\cos (y\sqrt3)-\cos x)\sin x.$$ 
A scaled version of this function was used in one of the  examples of
\cite{P:symm-ell}; in fact, $w=u_x$,  where
 $u(x,y)= (\cos x - \cos(\sqrt 3 y))^2/2$
 is a nonnegative solution of  \eqref{eeq}, \eqref{ebc} with 
$h(y)=-4 \sin^2(\sqrt 3  y)$ and
$\Om=\{(x,y)\in\R^2\;:\;|x-2\pi\pm\sqrt 3 y|<2\pi\}$.
As depicted in Figure~\ref{figa}, 
the nodal set of $w$ in $\bar \Om$ consists of line segments 
which intersect  at degenerate zeros of $w$.  Our goal is to
perturb $w$, adding to it a small multiple of another solution of $\De
v+4 v=0$, so as to deform the nodal structure in Figure~\ref{figa} to
that in Figure~\ref{fig2}. Thus, after the perturbation
the solution looses some of its degenerate zeros, producing smooth nodal
curves near the original corners of $\Om$,  while other degenerate
zeros are kept intact to make intersections of nodal curves with
$\partial \Om$ possible. The details of this perturbation analysis are
given in the next section, together with the proof of Theorem \ref{thm}.

      \begin{figure}[h] 
      \begin{center}
         \addtolength{\belowcaptionskip}{10pt}
          \includegraphics[scale=.2]{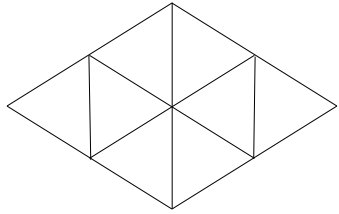}\qquad
         \caption[The nodal set of $w$]{\small The nodal set of 
            $w$ in  $\bar\Om$.\label{figa}}   
          \end{center} 
         \end{figure}

We remark that for the proof of Theorem \ref{thm}, it is not necessary
that $v$ vanishes on the whole boundary of $\Om$. However, this is the
case in our construction and it yields the  extra
information that the solution $u$ also satisfies
\begin{equation}
  \label{bcN}
  \frac{\partial u(x,y)}{\partial \nu}=0\quad (x,y)\in \partial\Omega,
\end{equation}
where $\nu$ is the outer unit normal vector field on $\Om$. Thus 
$u$ is a nonzero nonnegative solution of the overdetermined 
problem \eqref{eeq}, \eqref{ebc}, \eqref{bcN} (note that \cite{Serrin:symm}
rules out the existence of such solutions for spatially homogeneous
equations, unless $\Om$ is a ball). 

Finally, we remark that the domain $\Om$ in our construction is also
convex in $y$ and symmetric about the $x$-axis, and the function
 $h$ is an even function of $y$. Thus the only obstacle to a possible application of
the method of moving hyperplanes in the $y$-direction  
(which, obviously, would rule out the nodal
structure in Figure~\ref{fig1}) is the fact that $h$ is not decreasing in
$y>0$.

\section{Proof of Theorem \ref{thm}}
Following the above outline, we first want to find a solution of the linear
equation
\begin{equation}
  \label{eqlin}
  \De v+4 v=0
\end{equation}
with a suitable nodal structure.
To start with, we consider the function 
\begin{equation}
  \label{wdef}
  w(x,y):=(\cos (y\sqrt3)-\cos x)\sin x.
\end{equation}
 It is a solution of \eqref{eqlin} on $\R^2$ whose  nodal set  consists of
the lines
\begin{align}
&\{x=k\pi\}, \ k\in \Z, \quad\text{and}\quad \label{nodlinesx}\\
&\{y=\pm\frac{1}{\sqrt 3}(x+2k\pi)\},\ k\in \Z. \label{nodlinesxy}
\end{align}

Moreover, $w$ is odd about each of these nodal lines
 and it is even about  the horizontal lines $y=k\pi/\sqrt3$, $k\in \Z$.
We say that $w$  is even (resp. odd) about a line, if $w=w\circ P$
(resp. $w=-w\circ  P$),
where $P$ is   the reflection  about the  line in question.

Our goal is to find a perturbation $w+\ep \psi$ of  $w$
with the  nodal structure as depicted in Figure~\ref{fig2}.   This will be accomplished by means
of  a function $\psi:\R^2\to\R$ 
with the following properties:
\begin{itemize}
 \item[(W1)]\quad {$\psi$ is a solution of \eqref{eqlin} on
     $\R^2$,}
\quad 
 \item[(W2)]\quad
$\psi(k\pi,\cdot)\equiv 0$ and $\psi $  
is odd about the vertical line $\{x=k\pi\}$ 
 for each $k\in \Z $,
 \item[(W3)]\quad
 $\psi$  is even about the $x$-axis,   
\item[(W4)]\quad $D\psi(z_0)=0,\quad D^2\psi(z_0)=0$,  
\item[(W5)]\quad $\psi_x(z_1)<0$,\quad  $\psi_x(z_2)>0$,\quad $\psi_{xy}(z_2)>0$,
\end{itemize}
where 
\begin{equation}
  \label{zdef}
 z_0=(\pi
,\pi/\sqrt3),\quad z_1=(0,0),\quad z_2=(0,2\pi/\sqrt3), 
\end{equation}
and $D\psi$ and $D^2\psi$ stand for the gradient and the Hessian
matrix of $\psi$, respectively. 
Note that  $z_0$, $z_1$, $z_2$  are the only degenerate zeros
of $w$ in $[0,\pi]\times [0,2\pi/\sqrt3)$ (see Figure~\ref{figb}).
\begin{figure}[h]
         \addtolength{\belowcaptionskip}{10pt}
         \addtolength{\abovecaptionskip}{-2cm}  
         \begin{center}
\hspace{-2cm} \includegraphics[scale=.35]{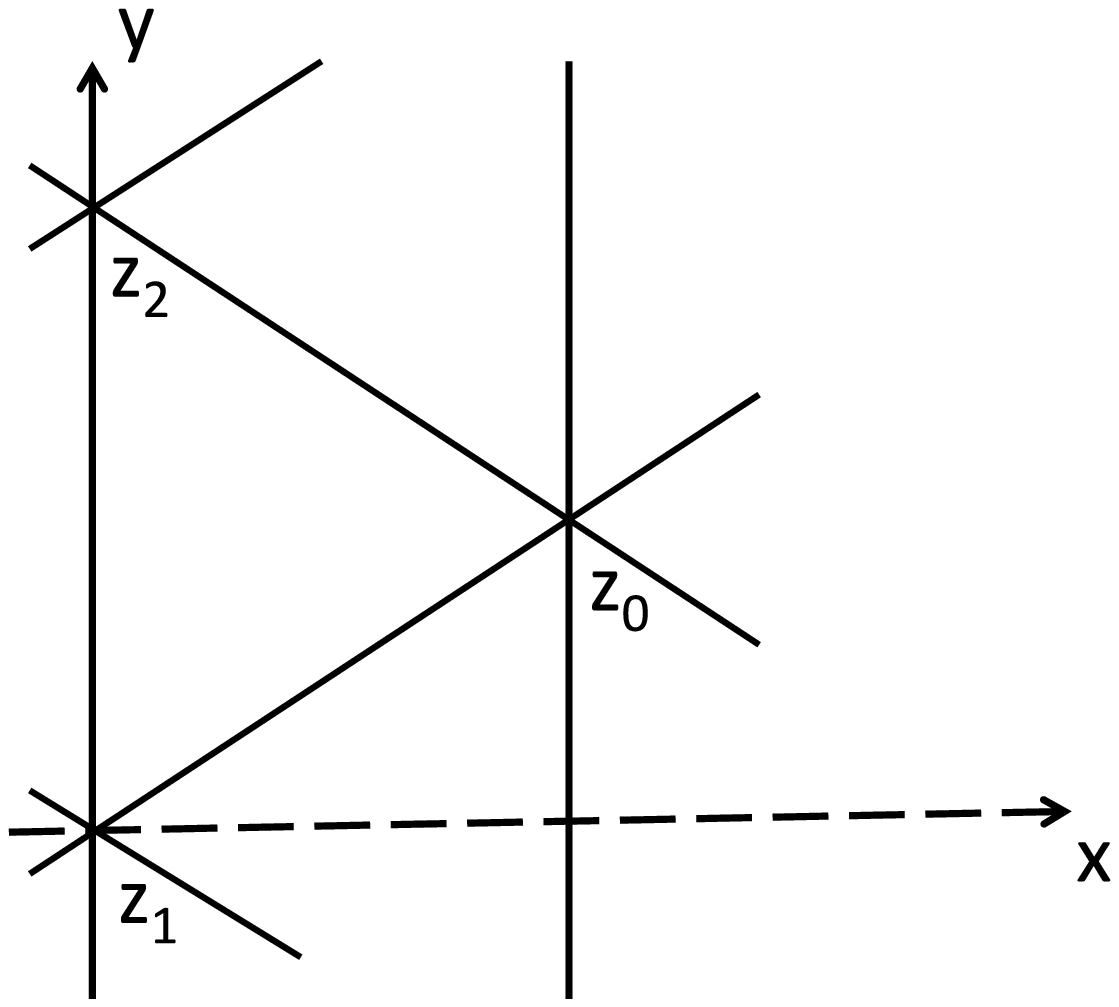}\hspace{-2.5cm}
\includegraphics[scale=.35]{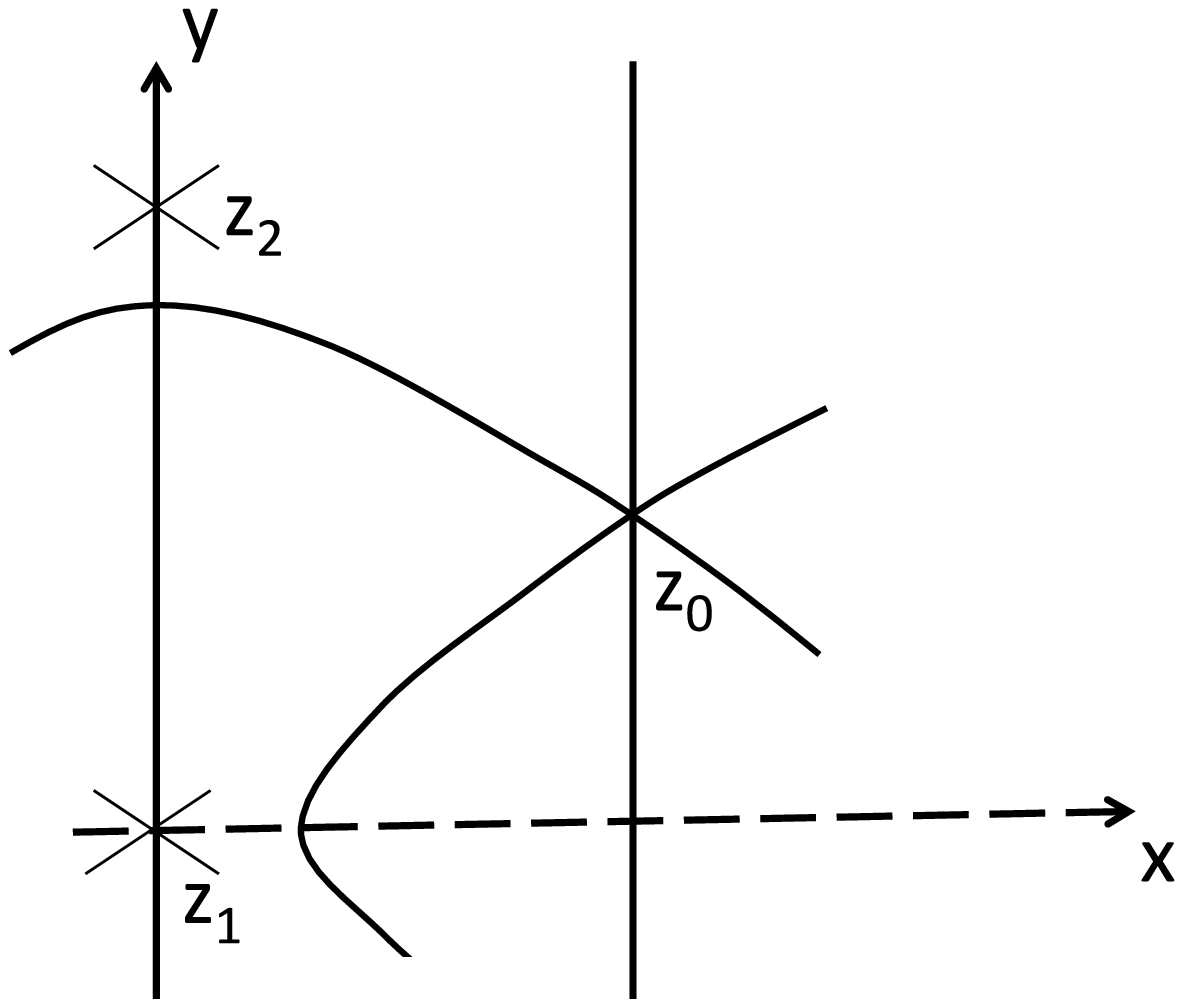}\hspace{-2.5cm}
        \caption[]{\small The solid lines in the left figure show
           the nodal lines of $w$ intersecting at the
          degenerate zeros $z_0$, $z_1$, $z_2$. The right figure depicts
          the effect of the perturbation $w+\ep \psi$ on the
          nodal set near  $z_0$, $z_1$, $z_2$, and in $(0,\pi)\times
          (0,2\pi/\sqrt3)$  under assumptions
          \rm(W1)--(W5) (Lemma \ref{lewv}).  The 
          whole domain $\Omega$
          can be recovered from the right picture by performing, in
          succession,  reflections about the lines $x=\pi$, $y=0$, and
           $x=0$. 
         \label{figb}   }
          \end{center}
         \end{figure} 

\begin{lemma}
  \label{lepsi}
There exist is an analytic
function $\psi:\R^2\to \R$ such that {\rm(W1)--(W5)} hold.
\end{lemma}
We postpone the proof of this lemma until the end of this section.

\begin{lemma}
  \label{lewv}
Assume that  $\psi:\R^2\to \R$ is an analytic function satisfying
 {\rm(W1)--(W5)}. If $\ep>0$ is sufficiently small, then the
function $v=w+\ep \psi$ has the following properties:
\begin{itemize}
 \item[\rm (V1)]\quad {$v$ is a solution of \eqref{eqlin} on $\R^2$,}
 \item[\rm (V2)]\quad $v(k\pi,\cdot)\equiv 0$ and
$v$  is odd about the vertical line $\{x=k\pi\}$ for each $k\in \Z $,
 \item[\rm (V3)]\quad
 $v$  is even about the $x$-axis. 
\item[\rm (V4)] \quad There exist $s\in (\pi/\sqrt3,2 \pi/\sqrt3 )$ 
  and a continuous  function $\mu$
  on $[-s,s]$ with the following properties:
  \begin{itemize}
  \item[(i)] $\mu$ is even, $0<\mu<2\pi$ on $[0,s)$, $\mu(s)=0$, and 
  $\mu(\pi/\sqrt3)=\pi$,
   \item[(ii)] $\mu$ is analytic in $(-s,s)$,  $\mu'<0$  on $(0,s)$,
    and $\mu(y)\mu'(y)$ has a finite limit as $y\upto s$, 
   \item[(iii)] the domain $\Om:=\{(x,y): y\in (-s,s), -\mu(y)<x<\mu(y)\}$
       is analytic,
   \item[(iv)] the nodal set of $v$ in $\bar \Om$ consists of
     $$\partial \Om= \{( m(y),y): y\in [-s,s]\} \cup \{(-
     m(y),y): y\in [-s,s]\},$$
   the line segments 
    $\{(x,y)\in \Om: x =k\pi\}$, $k=0,\pm1$, and the two analytic curves
    $\{(2\pi -\mu(y),y): y\in (-\pi/\sqrt3, \pi/\sqrt3])\}$, 
   $\{(-2\pi +\mu(y),y): y\in (-\pi/\sqrt3, \pi/\sqrt3])\}$.
  \end{itemize}
\end{itemize}
\end{lemma}
Note that according to (V4), the nodal set of $v$ is as in Figure~2. 

\begin{proof}[Proof of Lemma \ref{lewv}.]
  Properties (V0)-(V3) follow immediately from (W1)-(W3)
  (independently of the choice of $\ep$). Let us now consider the nodal set
  of $v$ in $[-2\pi,2\pi]\times [-2\pi/\sqrt3, 2\pi/\sqrt3]$.  By the
  symmetry properties (V2), (V3), we only need to understand the
  nodal set in $[0,\pi]\times [0,2\pi/\sqrt3]$; the rest is determined
  by reflections. 
   We first
  investigate the nodal set of $v$ near the degenerate zeros
  $z_0$, $z_1$, $z_2$ of $w$.

\vspace{8pt}
\emph{Local analysis near  $z_0=(\pi,\pi/\sqrt3)$.} 
By  (W2), $v$ is odd about the vertical line $x=\pi$, in
particular, $v(z_0)=0$. By \eqref{wdef} and (W4),  
\begin{equation}
  \label{zdeg}
  Dv(z_0)=0,\quad      D^2v(z_0)=0.
\end{equation}

We next apply to $v$ the
following well-known equal-angle property
of the nodal set of solutions of a planar linear elliptic equations
(see, for example, \cite{Alessandrini,Hartman-W} or \cite[Theorem
2.1]{Terracini-H-HO}). From such well known results, $v$ has a finite
order, say $j$, of vanishing at $z_0$. Moreover, there is a ball $B$ 
centered at $z_0$ such that 
the nodal set of $v$ in $B$ consists of  $k:=2j$ $C^1$-curves ending at
$z_0$ and having  
tangents at $z_0$, and these  tangents form  $k$ angles of equal
size. In the present case,  relations \eqref{zdeg} imply $j\ge 3$, hence
$k\ge 6$.

We claim that if $\ep$ is sufficiently small,
then $k=6$. Indeed, since $\psi$ is odd about the line $x=\pi$, we
can write it as $\psi(x,y)=\tilde \psi(x,y) \sin x$, where $\tilde
\psi$ is an analytic function, which is even about $x=\pi$. 
Then also $v(x,y)=\tilde
v(x,y)\sin x$, where  $\tilde v$ is still an
analytic function, which is even about $x=\pi$. 
A simple computation
shows that $D\tilde v(z_0)=0$ and, for small $\ep>0$, $D^2\tilde
v(z_0)\approx \diag(-1,3)$. The Morse lemma implies that the nodal set
of $\tilde v$ near $z_0$ consists of two smooth curves transversally intersecting at
$z_0$. These can be viewed as four curves ending at  $z_0$, which 
together  with  two segments of the vertical line $\{x=\pi\}$ 
exhaust the nodal set of $v$ near  $z_0$. This gives   $k\le 6$, hence
$k=6$ as claimed. 

Now,   the fact that $v$ is odd about $x=\pi$, in conjunction
with the equal angle condition, implies the following conclusion.

\begin{list}{{\bf (C0)}} { \usecounter{lcount}
\setlength{\labelwidth}{10pt}
\addtolength{\leftmargin}{-10pt}
}
\item If  $r_0>0$ is a sufficiently small radius, then in the ball
  $B(z_0,r_0)$ the set $v^{-1}(0)$  consists of the vertical line 
segment $\{(x,y)\in B(z_0,r_0):x=\pi\}$ and two smooth curves $\Ga_1$
and  $P_\pi(\Ga_1)$, where  $P_\pi$ denotes the  reflection about 
the line $\{x=\pi\}$. The two curves intersect at $z_0$,  $\Ga_1$ 
is tangent at   $z_0$ to $(\sqrt3,-1)$, hence (with small enough
$r_0$) at each of its points, $\Ga_1$ is tangent to a vector in 
$\{(x,y): x>0,y<0\}.$
\end{list}
We shall presently see  that the curve $\Ga_1$ in this conclusion is
actually analytic. By the evenness about $\{x=\pi\}$, $\tilde
v(x,y)=\varphi((x-\pi)^2,y)$, where $\varphi(q,r)$ is an analytic function
near $(q,r)=(0,\pi/\sqrt3)$. For small $\ep>0$, we have 
$\varphi_q(0,\pi/\sqrt3)\ne 0$, hence the zeros of $\varphi$ near
$(0,\pi/\sqrt3)$ are given by $q=a(y)$, where $a$ is an analytic
function satisfying $a(\pi/\sqrt3)=0$. 
Then the nodal set of $\tilde v$ near $z_0$ is given by 
the equation $(x-\pi)^2 =a(y)$. Since we already know that the nodal
set consists of two smooth curves, it is easy to verify they must be
analytic.

\vspace{8pt}
\emph{Local analysis near    $z_1=(0,0)$.} 
 In view of (W2) and (W3), in a neighborhood of $z_1$ we
have $\psi(x,y)=\tilde \psi(x,y) \sin x$, where $\tilde
\psi(x,y)$ is an analytic function, which is even about the coordinate
axes. 
Denote 
$$\tilde w(x,y):=\cos (y\sqrt3)-\cos x,\quad \tilde v:=
\tilde w +\ep \tilde \psi.$$
By (W5),  $\tilde \psi(0,0)=\psi_x(0,0)<0$.
Further,  $\tilde w(0,y)=\cos (y\sqrt3)-1\le 0$  
and $\tilde w_{xx}(x,y)=\cos x >0$ if
$x\approx 0$. Therefore,  
there exist positive constants $\al<\pi$, $\be<\pi/\sqrt 3$,
such that   $\tilde \psi<0<\tilde w_{xx}$ in the rectangle
$[0,\al]\times [-\be,\be]$   (this is true
regardless of $\ep$, as long as  $\ep>0$). Consequently, 
$\tilde v<0$ on the segment $\{0\}\times  [-\be,\be]$ and, 
if $\ep>0$ is sufficiently
small, also  $\tilde
v_{xx}>0$ in $[0,\al]\times [-\be,\be]$. The latter and the 
relation $\tilde
v_{x}(0,y)=0$ (which follows from the 
evenness of $\tilde v$) imply that $\tilde
v_{x}>0$ in $(0,\al]\times [-\be,\be]$. Finally, making $\be>0$
smaller if necessary, we have $\tilde w>0$ on the segment
$\{\al\}\times[-\be,\be]$, hence $\tilde v>0$ on that segment
if $\ep>0$ is sufficiently small. 
We conclude, that if $\ep>0$ is sufficiently small, then for each $y\in
[-\be,\be]$, the function $\tilde v(\cdot,y)$ has a unique zero
$x=\xi(y)$ in $(0,\al)$. By the implicit function theorem, the
function $\xi$ is analytic and, by the uniqueness, $\xi$ is even. 

We now show that $\xi'(y)>0$ for $y>0$. 
Differentiating the identity $\tilde v(\xi(y),y)=0$, we obtain
$\tilde v_x\xi' +\tilde v_y\equiv 0$. Since $\tilde v_x>0$ in 
$(0,\al]\times [-\be,\be]$,  we need to show
that $\tilde v_y<0$ in $(0,\al]\times (0,\be]$. By the evenness about
$\{y=0\}$, $\tilde v_y=0$ when $y=0$. Since $\tilde v_{yy}(x,y)=-3\cos
(y\sqrt3)+\ep\tilde \psi_{yy}$, making $\be$ smaller, if necessary, we
achieve that  $\tilde v_{yy}<0$ in $(0,\al]\times (0,\be]$, 
for all sufficiently
small $\ep>0$. This gives 
$\tilde v_{y}<0$ in $(0,\al]\times (0,\be]$, as desired. 

We summarize that for some positive constants $\al<\pi$,
$\be<\pi/\sqrt 3$, the following statement is valid:

\begin{list}{{\bf (C1)}} { \usecounter{lcount}
\setlength{\labelwidth}{10pt}
\addtolength{\leftmargin}{-10pt}
}
\item For all  sufficiently small $\ep>0$, the nodal set of $v$ in 
$[0,\al]\times [-\be,\be]$ consists of the segment $\{0\}\times
 [-\be,\be]$ and the
 curve $\Ga_2:=\{(\xi(y),y):y\in [-\be,\be]\}$, where
 $\xi:[-\be,\be]\to (0,\al)$ is an even 
 analytic function with $\xi'>0$ on $(0,\be]$. 
\end{list}

\vspace{8pt}
\emph{Local analysis near  $z_2=(0,2\pi/\sqrt3)$.} We proceed 
similarly as  in the previous analysis. In a neighborhood of $z_2$, we
have $\psi(x,y)=\tilde \psi(x,y) \sin x$, where $\tilde
\psi(x,y)$ is an analytic function, which is even about  $\{x=0\}$. We
set
$$\tilde w(x,y):=\cos (y\sqrt3)-\cos x,\quad \tilde v:=
\tilde w +\ep \tilde \psi.$$ 
The functions $\tilde v$ and $\tilde \psi$ are
even about the $y$-axis. By  (W5), 
\begin{equation*}
  \text{$\tilde \psi(0,2\pi/\sqrt3)= \psi_x(0,2\pi/\sqrt3) >0$ and 
 $\tilde \psi_y(0,2\pi/\sqrt3)= \psi_{yx}(0,2\pi/\sqrt3) >0$.}
\end{equation*}
 Further, $\tilde w(x,2\pi/\sqrt3 )=1-\cos x\ge 0$ and 
 \begin{equation*}
   \text{
$\tilde w_y(x,2\pi/\sqrt3 )=-\sqrt3 \sin (y\sqrt3) \ge 0$ for $y\le
2\pi/\sqrt3$,  $y\approx
2\pi/\sqrt3$.  } 
 \end{equation*}
Therefore,  
there exist positive constants $\ga,  \de$
such that  $\tilde v>0$ on the segment
  $[-\ga,\ga]\times
\{2\pi/\sqrt3\}$
and $\tilde v_y>0$ in  $[-\ga,\ga] \times [2\pi/\sqrt3-\de,2\pi/\sqrt3]$
(this is true for each  $\ep>0$).
Making $\ga$ smaller if necessary, we also have 
 $\tilde w<0$ on the segment
$[-\ga,\ga] \times \{2\pi/\sqrt3-\de\}$,  
hence $\tilde v<0$ on that segment if $\ep>0$ is sufficiently small. 
Thus for each small $\ep>0$ and for each $x\in  [-\ga,\ga]$, the function
$\tilde v(x,\cdot)$  has a unique zero $\eta(x)$ 
in $(2\pi/\sqrt3-\de,2\pi/\sqrt3)$ and
the function $\eta$ is  analytic and even. 
We shall show in a moment that, possibly after making $\ga>0$ smaller, 
for all sufficiently small $\ep>0$
one has $\eta''(0)<0$ and $\eta'<0$
on $(0,\ga]$. 
Therefore, with $s:=\eta(0)$, 
the following conclusion is valid.

\begin{list}{{\bf (C2a)}} { \usecounter{lcount}
\setlength{\labelwidth}{10pt}
\addtolength{\leftmargin}{-10pt}
}
\item If $\ep>0$  is sufficiently small, the nodal set of $v$ in 
$[-\ga,\ga] \times [2\pi/\sqrt3-\de,2\pi/\sqrt3]$ consists of the
vertical segment $\{0\}\times[2\pi/\sqrt3-\de,2\pi/\sqrt3]$ and the
 curve $\Ga_3:=\{(x,\eta(x)): x\in  [-\ga,\ga]\}$. Here 
 $\eta:[-\ga,\ga]\to (2\pi/\sqrt3-\de,2\pi/\sqrt3)$ is an even 
 analytic function with $\eta''(0)<0$ and $\eta'<0$ on
 $(0,\ga)$.
\end{list}

To verify that, indeed, $\eta'<0$ for all sufficiently small $\ep$,
differentiate   the identity $\tilde v(x,\eta(x))=0$. This gives
$\tilde v_x+\tilde v_y\eta' \equiv 0$ and since $\tilde v_y>0$ in 
$[-\ga,\ga] \times [2\pi/\sqrt3-\de,2\pi/\sqrt3]$, it is sufficient to
 verify that  $\tilde v_x>0$ in $(0,\ga] \times [2\pi/\sqrt3-\de,2\pi/\sqrt3]$. 
We have $\tilde v_x=0$ when $x=0$ (by the evenness) and 
$\tilde v_{xx}(x,y)=\cos x+\ep \tilde \psi_{xx}(x,y)$. Thus, making $\ga$
smaller, if necessary, we have $\tilde v_{xx}>0$ in 
$[0,\ga] \times [2\pi/\sqrt3-\de,2\pi/\sqrt3]$ for each sufficiently small
$\ep>0$. This gives $\tilde v_x>0$ in $(0,\ga] \times
[2\pi/\sqrt3-\de,2\pi/\sqrt3]$, as needed. Also
$\eta''(0)=-\tilde v_{xx}(0,s)/\tilde v_y(0,s)<0$ for all sufficiently
small $\ep>0$. 

Below it will be useful to have introduced  the inverse function to  
 $\eta\rest_{(0,\ga)}$. Using  (C2a), elementary arguments verify the
 following statements. 
\begin{list}{{\bf (C2b)}} { \usecounter{lcount}
\setlength{\labelwidth}{10pt}
\addtolength{\leftmargin}{-10pt}
}
\item With $\eta$ as in in (C2a), the function 
$\zeta:= (\eta\rest_{(0,\ga)})^{-1}:  (\eta(\ga),s)\to (0,\ga)$
is analytic, $\zeta'<0$ on $(\eta(\ga),s)$ and  $\xi(y)\xi'(y)$ 
has a finite limit as $y\upto s$.
\end{list}

\vspace{5pt} 

We now give a global description of the nodal set of
 $v$ in $[0,\pi]\times [0,2\pi/\sqrt3]$.  Since $v$ is an analytic
 function,  
the implicit
  function theorem implies that away from the degenerate zeros of 
  $v$, the nodal set of $v$  consists of
  analytic curves. Using the explicit structure of the 
   nodal set of $w$, as  given in \eqref{nodlinesx},
   \eqref{nodlinesxy}, and the fact that $z_0$, $z_1$, $z_2$  are the only
  degenerate zeroes of $w$ in  $[0,\pi]\times [0,2\pi/\sqrt3)$, a simple
  continuity argument leads to the  the following conclusion.  
\begin{list}{{\bf (CG)}} { \usecounter{lcount}
\setlength{\labelwidth}{10pt}
\addtolength{\leftmargin}{-10pt}
}
\item For any $r>0$,
  there is $\ep_0>0$ such that for each $\ep\in (0,\ep_0)$ 
  the nodal set  of $v$ in 
$$G:= [0,\pi]\times [0,2\pi/\sqrt3]\setminus 
  \bigcup_{i=1,2, 3} B(z_i,r)$$
consists of segments of the vertical lines $\{x=0\}$, $\{x=\pi\}$,
and two analytic curves 
$\Ga_4$, $\Ga_5$,  $C^1$-close to the
line segments  $\{(x,y)\in G: x=2\pi-y\sqrt3\}$, 
$\{(x,y)\in G: x=y\sqrt3\}$, respectively.
In particular, $\Ga_4$ is at each of its points
tangent to a vector $\{(x,y): x>0, y<0\}$
 and $\Ga_5$  is at each  point
tangent to a vector in $\{(x,y): x>0,
y>0\}$. 
\end{list}

 To complete the proof of Lemma \ref{lewv}, we choose
 $r>0$ smaller than each of the positive 
 constants $r_0$,  $\al$, $\be$, $\ga$,
$\de$ appearing in  (C0)-(C2), so that
 \begin{align*}
  \bar  B(z_0,r)&\subset B(z_0,r_0), \\
  \bar B(z_1,r)&\subset 
( -\al,\al)\times (-\be,\be), \\ 
 \bar B(z_2,r)&\subset 
( -\ga,\ga)\times (-\de,\de). 
 \end{align*}
Then, by  (C0) and (C2a),
\begin{equation}
  \label{eq:1}
  \begin{aligned}
 \Ga_4 \cap \left((0 ,\ga)\times (2\pi/\sqrt3-\de,2\pi/\sqrt3)\right)& \subset
\Ga_3,\\
\Ga_4\cap B(z_0,r)& \subset \Ga_1.
\end{aligned}
\end{equation}
Similarly, by  (C0) and (C1),
\begin{equation}
  \label{eq:2} 
\begin{aligned}
 \Ga_5\cap B(z_0,r)& \subset P_\pi(\Ga_1),\\
\Ga_5\cap\left((0,\al)\times (0,\be)\right)& \subset \Ga_2,
\end{aligned}
\end{equation}
which is equivalent to 
\begin{equation}
  \label{eq:3} 
\begin{aligned}
 P_\pi( \Ga_5\cap B(z_0,r))& \subset \Ga_1,\\
 P_\pi( \Ga_5\cap((0,\al)\times (0,\be)))& \subset  P_\pi(\Ga_2).
\end{aligned}
\end{equation}

By \eqref{eq:1}, \eqref{eq:3}, the union
$$
\Ga:=\Ga_3\cup \Ga_4\cup \Ga_1 \cup P_\pi(\Ga_5)\cup P_\pi(\Ga_2)
$$
is an analytic curve. 
Moreover,  (C0)-(CG) imply that at each point of $\Ga\cap (
(0,2\pi)\times (0,2\pi/\sqrt3))$, $\Ga$ has a tangent vector in
$\{(x,y):x>0,y<0\}$. Therefore, there is an
analytic function
 $\mu:(0,s)\to (0,2\pi)$  such that  $\mu'<0$ on $(0,s)$, and 
\begin{equation*}
  \Ga\cap \left(
(0,2\pi)\times (0,2\pi/\sqrt3)\right)=\{(\mu(y),y):y\in (0,s)\}.
\end{equation*}
Clearly,  $(\mu(\pi/\sqrt3),\pi/\sqrt3)=z_0$, so
$\mu(\pi/\sqrt3) =\pi$. Moreover, 
near $y=0$, $\mu$ coincides with the function 
$P_\pi\circ \xi$ and near $y=s$ it coincides with the function
$\zeta$ (see (C1), (C2b)). Therefore, $\mu$ extends to a continuous
even function on $[-s,s]$,  analytic in $(-s,s)$, 
which satisfies   statements (i),(ii) of
(V4). Define $\Om$ as in (V4)(iii). Since $\mu\equiv
\zeta=(\eta\rest_{(0,\ga)})^{-1}$ near $y=s$ 
and $\eta$ is an even analytic
function, $\Om$ is an analytic domain.
   Finally, since $v$ is odd about 
 $\{x=k\pi\}$, $k\in\Z$, the curve $\Ga$ and the whole
boundary $\partial \Om$ belong to the nodal set of $v$. The oddness
of $v$ and the global description of the nodal set of $v$ in 
$[0,\pi]\times [0,2\pi/\sqrt3]$, as given above, imply that the nodal
set of $v$ in $\Om$ is as stated in  (V4)(iv). 
This completes the proof of Lemma \ref{lewv}.
\end{proof}

We next prove that Theorem \ref{thm} follows from Lemma \ref{lewv}.

\begin{proof}[Proof of Theorem \ref{thm}]
  Let $v$, $\mu$, and $\Om$ be as in Lemma \ref{lewv}. Then $\Om$ is an
  analytic domain, which is convex in $x$ and symmetric about the
  $y$-axis (it is also convex in $y$ and symmetric about the
  $x$-axis).  Replacing $v$ with $-v$, we can assume that 
  \begin{equation}
    \label{eq:4}
    \text{$v>0$ in   $\{(x,y)\in \Om:x\le -\pi$\}, }
  \end{equation}
  which is the left-most nodal domain of $v$. 
 For each $(x,y)\in \bar \Om$ we define
 \begin{equation*}
   u(x,y):=\int_{\mu(y)}^xv(t,y)\,dt. \quad
 \end{equation*}
 Then $u$ is analytic in $\Om$ and continuous on $\partial \Om$.
 Since $v$ is odd about the lines $\{x=k\pi\}$, $k\in \{-1,0,1\}$, $u$
vanishes on $\partial \Om$ and on 
$\{(2\pi -\mu(y): y\in (-\pi/\sqrt3, \pi/\sqrt3])\}$, 
   $\{(-2\pi +\mu(y): y\in (-\pi/\sqrt3, \pi/\sqrt3])\}$. 
From \eqref{eq:4} and the oddness properties of
 $v$ it follows that $u>0$ in the rest of $\Om$. 

Next, we compute (using $v(\mu(y),y)=0$) 
\begin{align*}
  \Delta u(x,y)&=v_x(x,y)+\int_{\mu(y)}^x v_{yy}(s,y)\,ds-
  v_{y}(\mu(y),y)\mu'(y)\quad \\
&=v_x(x,y)-\int_{\mu(y)}^x(v_{xx}(s,y)+4
v(s,y))\,ds-v_{y}(\mu(y),y)\mu'(y)\\
&=v_x(\mu(y),y)-4 u(x,y)-v_{y}(\mu(y),y)\mu'(y).
\end{align*} 
This shows that $u$ solves  
\eqref{eeq} with $h(y):=v_{y}(\mu(y),y)\mu'(y)-
v_x(\mu(y),y)$. Clearly, $h$ is even and continuous (in fact, analytic) in
$(-s,s)$. Since $v_{y}(0,y)=0$ and $v_y$ is analytic, (V4)(ii) implies
that $h(y)$ has a finite limit as $y\upto s$. Therefore $h$ extends to
an even continuous function on $\R$. 
\end{proof}

It remains to prove Lemma \ref{lepsi}.
\begin{proof}[Proof of Lemma \ref{lepsi}] We look for $\psi$ in the
  form
  \begin{equation}
    \label{eqpsi}
    \psi(x,y)=\sum_{k\in A}c_k\sin(kx)\cosh(y\sqrt{k^2-4}), 
  \end{equation}
  where $A$ is a finite subset of $\{k\in \N:k>4\}$ and $c_k$ are real
  coefficients 
  to be determined. Obviously, this function satisfies (W1)-(W3) and
  \begin{equation*}
    \psi_y(\pi,y)=\psi_{xx} (\pi,y) =\psi_{yy} (\pi,y)=0\text{ at
      $y=\pi/\sqrt3$}. 
  \end{equation*}
  To meet the remaining requirements in (W4), (W5), we postulate
  \begin{equation}
    \label{systc}
     \begin{aligned}
 \psi_x(0,0)   &=-1,\quad \psi_x(\pi,\pi/\sqrt3)=0,\quad \psi_{xy}(\pi ,\pi/\sqrt3)=0,  \\
   &\psi_x(0,2\pi/\sqrt3)=1,\quad 
    \psi_{xy}(0,2\pi/\sqrt3)=1.
  \end{aligned}
  \end{equation}
Substituting \eqref{eqpsi} into \eqref{systc}, we obtain a system of
five equations to be solved for $c_k$, $k\in A$. If 
$A$ consists of five even  integers $k_1<\dots<k_5$, then 
\eqref{systc} is solvable if $\det M\ne 0$ for the $5\times 5$ 
 matrix $M$ whose rows are given by
 \begin{align*}
   (1,1,1,1,1), \ &\left(\cosh \frac{\pi\nu_j}{\sqrt3} \right)_{j=1}^5, 
\ \left(\nu_j\sinh \frac{\pi\nu_j}{\sqrt3} \right)_{j=1}^5, \\
\ &\left(\cosh \frac{2\pi\nu_j}{\sqrt3} \right)_{j=1}^5,
\ \left(\nu_j\sinh \frac{2\pi\nu_j}{\sqrt3} \right)_{j=1}^5, 
 \end{align*}
where  $\nu_j:=\sqrt{k_j^2-4}$, $j=1,\dots 5$. 
We want to select the $k_j$, inductively, 
such that the leading principal minors
(further just the minors) of $M$ have nonzero determinants. The first
minor has determinant 1.  Assume that for some $j<5$, 
 $k_1<\dots <k_j$  have been selected such that the 
$j$-th minor has nonzero determinant. 
Consider the  determinant of the
$j+1$-st minor as a function of $\nu_{j+1}$.
Expand this determinant down its $j+1$-st column. The last term in
this expansion has the fastest growth, as $\nu_{j+1}\to \infty$, and 
it is multiplied by a nonzero constant (by the induction
hypothesis). Hence if $k_{j+1}$ is sufficiently large, the
  determinant of the $j+1$-st minor is nonzero. 
This completes the induction, showing that 
the even numbers $k_j$ can be selected so as to make $\det M\ne 0$.

Thus \eqref{systc} can be solved for $c_{k_1},\dots, c_{k_5}$ and the
resulting function $\psi$ then satisfies all statements (W1)-(W5). 
\end{proof}

\bibliographystyle{amsplain} 

\providecommand{\bysame}{\leavevmode\hbox to3em{\hrulefill}\thinspace}
\providecommand{\MR}{\relax\ifhmode\unskip\space\fi MR }
\providecommand{\MRhref}[2]{%
  \href{http://www.ams.org/mathscinet-getitem?mr=#1}{#2}
}
\providecommand{\href}[2]{#2}

\end{document}